\documentclass{amsart}
\usepackage[latin1]{inputenc}
\usepackage[active]{srcltx}
\usepackage{fullpage}
\usepackage{amsfonts,enumerate}
\usepackage[mathscr]{euscript}
\usepackage{epsfig}
\usepackage{latexsym}
\usepackage[all]{xy}
\usepackage{amssymb}
\usepackage{amsthm}
\usepackage{amsmath}
\usepackage{amsxtra}
\usepackage[dvipsnames]{xcolor}
\usepackage[pagebackref]{hyperref}
\usepackage{fancyvrb}
\hypersetup{citecolor=blue}
\usepackage[textwidth=2cm]{todonotes}
\usepackage{float}
\hypersetup{citecolor=MidnightBlue,
	colorlinks=true,
	linkcolor=MidnightBlue,
	urlcolor=MidnightBlue}

.tfm
.tfm

\theoremstyle{plain}
\newtheorem{prop}{Proposition}[section]
\newtheorem{thm}[prop]{Theorem}

\newtheorem{lemma}[prop]{Lemma}

\newtheorem*{thm*}{Theorem}
\newtheorem*{lemma*}{Lemma}
\newtheorem*{prop*}{Proposition}
\newtheorem*{thmA}{Theorem~\ref{thm:d=2}}
\newtheorem*{thmB}{Theorem~\ref{thm:d=3.1}}
\newtheorem*{thmC}{Theorem~\ref{thm:d=3.2}}
\newtheorem*{thmD}{Theorem~\ref{thm:dcong}}

\theoremstyle{definition}

\theoremstyle{remark}

\newtheorem{remark}{Remark}

\numberwithin{table}{section}

\newcommand{\B}{\mathcal B}

\newcommand{\Disc}{\Delta}

\newcommand{\rad}{\textup{rad}}

\def\RR{\mathbb R}

\def\<#1>{{\left\langle{#1}\right\rangle}}

\def\Z{{\mathbb Z}}             % integers
\def\Q{{\mathbb Q}}             % rationals
\def\id#1{{\mathfrak{#1}}}      % an ideal

\let\kro\dkro
\newcommand\TODO[1]{\textcolor{red}{#1}}
\newcommand{\mycomment}[1]{}

\begin{document}

	\title{On the sum of fifth powers in  arithmetic progression}

	\author{Lucas Villagra Torcomian}
	
	\address{FAMAF-CIEM, Universidad Nacional de C\'ordoba. C.P:5000,
		C\'ordoba, Argentina.}  \email{lucas.villagra@unc.edu.ar}

	\thanks{The author was supported by a CONICET grant.}

	\maketitle
	
	\vspace{-0.5cm}
	
	{\centering\footnotesize \textit{To Milena}\par}
	
	\vspace{-0.2cm}
	
	\begin{abstract}
		In this paper we study  equation $$(x-dr)^5+\cdots+x^5+\cdots+(x+dr)^5=y^p$$ under the condition $\gcd(x,r)=1$. %, as a natural continuation of the previous works \cite{BennettPatelSiksek,BennettKoutsianas,Koutsianasmuchos}. 
		We present a recipe for proving the non-existence of non-trivial integer solutions of the above equation, and as an application we obtain explicit results for the cases $d=2,3$ (the case $d=1$ was already solved \cite{BennettKoutsianas}). %We also state a new density result, proving an asymptotic result for $d\equiv 1, 7\pmod9$.
		We also prove an asymptotic result for $d\equiv 1, 7\pmod9$. Our main tools include the modular method, employing Frey curves and their associated modular forms, as well as the symplectic argument.
	\end{abstract}

	\section{Introduction}
	\subsection{Historical background}
	Let $f(x)$ be a polynomial with integer coefficients of degree at least 3, with no repeated roots. A result of Schinzel and Tijdeman \cite{SchinzelTijdeman} ensures that the superelliptic equation
	%\[f(x)=y^p, \quad x, y, p\in\Z\]
	\[f(x)=y^p\]
	has finitely many integer solutions with $p\ge2$ (where we count solutions with $y=\pm1,0$ only once). Replacing $f(x)$ by a binary form $f(x,r)$ of degree $k\ge3$, a result of Darmon and Granville \cite[Theorem 1]{DarmonGranville} implies that the generalized superelliptic equation
	%\[f(x,r)=y^p,\quad \gcd(x,r)=1,  \quad x,r,y,p\in\Z\] 
	\[f(x,r)=y^p,\quad \gcd(x,r)=1\] 
	has finitely many integer solutions, for a fixed exponent $p\ge\max\{2,7-k\}$. Evidence suggests that an analogous finiteness statement should hold taking $p$ as a free variable suitably large relative to $k$ (see \cite{BennettDahmen}), although a result of this kind seems to be still beyond current technology. %far to be proven 
	A nice example to illustrate its difficulty is mentioned in \cite{BennettDahmen}, noting that when $f(x,r)=xr(x+r)$, this would be equivalent to prove \textit{Asymptotic Fermat's Last Theorem}. On the other hand, there is a particular family of forms, arising from the sum of powers in arithmetic progression, which seems suitable for this problem. For a fixed integer $k$, let 
	\[f_j(x,r)=\sum_{i=0}^{j-1}(x+ir)^k.\]

	The origins of investigating perfect powers expressed as sums of powers in  arithmetic progression can be traced to Euler, when he observed in 1770 that $3^3+4^3+5^3=6^3$ \cite[art.\ 249]{Euler}. Since then, the study of powers as sum of consecutive cubes attracted the attention of many  mathematician over the 19th century, and their contribution can be found in \cite[pp.\ 582--588]{Dickson}. In 2013, Zhang and Bai \cite{ZhangBais} considered the generalized equation $f_j(x,1)=y^p$, giving rise to a special interest
	% \[f_j(x,1)=y^p,  \quad \quad k, j, x, y\in\ZZ. \]
	%\[x^k+(x+1)^k+\cdots+(x+d-1)^k=y^p,  \quad \quad k, d, x, y\in\ZZ. \]
	%In the last decades, special interest has been paid to 
	in perfect powers that are sum of consecutive powers (see \cite{Stroeker,BartoliSoydan, Soydan,ZhangBais, BennettPatelSiksek}). %A special case concerning this work is in \cite{BennettPatelSiksek}, where Bennett, Patel and Siksek completely solved the equation above when $k=5$ and $d=3$. 
	Moreover, in  recent years many results have appeared for the more general equation
	\begin{equation}\label{eq:GeneralMal}
		f_j(x,r)=y^p.
	\end{equation} %\begin{equation}\label{eq:GeneralMal}
	%	x^k+(x+r)^k+\cdots+(x+(d-1)r)^k=y^p, \quad \quad k, d, r, x, y\in\ZZ.
	%\end{equation}
	
	Patel and Siksek proved a density result, which implies that for even $k$, and $p\ge 2$,
	equation (\ref{eq:GeneralMal}) has no solutions
	for almost all $j \ge 2$ \cite{PatelSiksek}. As mentioned at the beginning, we are interested in the natural condition $\gcd(x,r)=1$. Under this assumption, a special case of interest is when $j$ is equal to an odd number $2d+1$, because in this case many powers of $x$ will vanish. Then, replacing $x$ by $x-dr$ in (\ref{eq:GeneralMal}) we get
	\begin{equation}\label{eq:GeneralMalImpar}
		(x-dr)^k+\cdots+x^k+\cdots+(x+dr)^k=y^p, \quad \gcd(x,r)=1.
	\end{equation}
	
	%We mention some results for   (\ref{eq:GeneralMalImpar}) when $k\ge 3$. 
	
	For $k=3$, $1\le r\le 10^6$ and $p\ge7$ a prime exponent, equation (\ref{eq:GeneralMalImpar}) was solved in the cases: $d=1$ \cite{Edis}, $d=2$ \cite{ArgaezGarcia}, $d=3$ \cite{GarciaPatel} and recently for $d=4$ \cite{todas}. %\TODO{Aca usan la hip de coprimos}. As mentioned at the beginning, we are interested in the natural condition $\gcd(x,r)=1$. Under this assumption, equation (\ref{eq:GeneralMalImpar}) was completely solved 
	For
	$(k,d)=(4,1)$ it was completely solved by van Langen \cite{VanLangen} and for $(k,d)=(5,1)$ by Bennett and Koutsianas, generalizing the results of \cite{BennettPatelSiksek}. %Instead of assuming $\gcd(x,r)=1$ but considering the condition $r=2^a5^b$ with integers $a,b\ge0$, the case $(k,d)=(5,1)$ was solved \cite{Koutsianasmuchos}.
	We refer to the recent survey \cite{todassurvey} for a detailed state of the art. %of (\ref{eq:GeneralMal}). %powers as sums of powers in \textcolor{red}{an} arithmetic progression. 
	Moreover, in \cite[Subsection 2.3]{todassurvey}, the authors suggest a list of open problems. %, e.g.\ solving  (\ref{eq:GeneralMalImpar}) for a fixed exponent $k\ge4$ with fixed $j$ and unrestricted arithmetic progression.  Our work is in this direction and it is a natural continuation of \cite{BennettPatelSiksek, BennettKoutsianas,Koutsianasmuchos} for $k=5$.
	The present article aims to address some of these questions for $k=5$, as a natural continuation of \cite{BennettPatelSiksek, BennettKoutsianas,Koutsianasmuchos}.
	
	\subsection{Methodology and main results}
	Let $r$ be a non-zero integer and  $d\ge1$. Let us consider the equation
	\begin{equation}\label{eq:k=5}
		(x-dr)^5+\cdots+x^5+\cdots+(x+dr)^5=y^p, \quad \gcd(x,r)=1.
	\end{equation}
	%In (\ref{eq:GeneralMalImpar}) we interpret \todo{$j=d$ (y mas arriba tambien). Y cambiar los valores de $d$ mas abajo}{$(k,d, r,p)$} as parameters and $(x,y)$ as a solution; we call the latter \textit{trivial} if it is equal to $(0,0)$.
	We call a solution $(x,y)$ of (\ref{eq:k=5}) \textit{trivial} if it is equal to $(0,0)$. Our goal is to give necessary conditions to prove the non-existence of non-trivial solutions of (\ref{eq:k=5}). The condition $\gcd(x,r)=1$ allows us to factorize the left hand side of (\ref{eq:k=5}) and reduce the problem to resolve generalized Fermat equations of signature $(4p,p,2)$. To do this, we follow a classical approach called the modular method, applied by Wiles to solve Fermat's Last Theorem \cite{Wiles}. This allow us to provide a recipe to study (\ref{eq:k=5}) for an arbitrary fixed value of $d$. As an application, we show how our approach works to get new results for the cases $d=2, 3$. These are the following.
	
	\begin{thmA}
		Let $r$ be a non-zero integer and $p>7$ be a prime. If 
		\[	(x-2r)^5+(x-r)^5+x^5+(x+r)^5+(x+2r)^5=y^p, \quad \gcd(x,r)=1,\]
		then $2\mid x$.
	\end{thmA}

	\begin{thmB}
		Let $r$ be a non-zero integer and $p>19$ be a prime. If $(x,y)$ is a non-trivial solution of 
		\begin{equation*}
			(x-3r)^5+(x-2r)^5+(x-r)^5+x^5+(x+r)^5+(x+2r)^5+(x+3r)^5=y^p,\quad \gcd(x,r)=1,
		\end{equation*}	
		then $2\nmid x$ and $35\mid x$. Moreover, if $\kro{-5}{p}=-1$ and in addition,  $\kro{2}{p}=1$ or $\kro{3}{p}=1$, then there are no non-trivial solutions.
		
	\end{thmB}

	\begin{thmC}
		Let $p>19$ be a prime and $r=2^{a_1}5^{a_2}7^{a_3}13^{a_4}$ with integers $a_i\ge0$. Then, there are no non-trivial solutions of equation
		\[	(x-3r)^5+(x-2r)^5+(x-r)^5+x^5+(x+r)^5+(x+2r)^5+(x+3r)^5=y^p, \quad \gcd(x,r)=1.\]
	\end{thmC}
	
	Based on the work of Patel and Siksek \cite{PatelSiksek}, a special problem suggested in \cite{todassurvey} is about getting a new density result for an odd exponent $k$. Inspired in \cite[Section 7]{BennettKoutsianas},  we prove that if $k=5$, then for $2/9$ of the integers $d$ we have an explicit bound of the exponent $p$, in order to have  non-trivial solutions of (\ref{eq:GeneralMalImpar}). More concretely, the result is as follows.

	\begin{thmD}
		Let $r$ be a non-zero integer,  $d\equiv 1,7\pmod9$ and $p\ge 2M_d$ be a prime number. Let $N_d$ be as in $(\ref{eq:Nd})$ and suppose that $$p>\left(\sqrt{\frac{\mu(N_d)}{6}}+1\right)^{\frac{N_d+1}{6}}.$$  Then, there are no non-trivial integer solutions of
		\[(x-dr)^5+\cdots+x^5+\cdots+(x+dr)^5=y^p, \quad \gcd(x,r)=1.
		\]
	\end{thmD}
	
	\vspace{0.7cm}
	The structure of the paper is as follows. In Section \ref{sec:modular} we present the modular method and we briefly recall the core of the idea of the symplectic argument, which results a combination of the modular method and a symplectic criterion. In Section \ref{sec:OnTheEquation} we address  equation (\ref{eq:GeneralMalImpar}) and, following the modular approach, we see how to attach a Frey elliptic curve to a solution of our Diophantine equation, applying the main result of this section, namely Proposition \ref{prop:ternaryeq}. In Section \ref{sec:applications} we prove, as an application,  the main results. All the computations were performed in the computer algebra package \verb|Magma| \cite{Magma}. The code can be found at
	\begin{center}
		\url{https://github.com/lucasvillagra/Fifth-powers.git}
	\end{center}

	\subsection*{Acknowledgments}
	I am very grateful to Ariel Pacetti for many useful conversations that allowed me to improve this work. I also thank Emilio Lauret for his support; this paper started as soon as he left my office.

	\section{The modular approach}
	\label{sec:modular}
	
	The modular method is one of the most powerful tools to approach exponential Diophantine equations. It has been successfully applied for the first time by Wiles to solve Fermat's Last Theorem \cite{Wiles}. As unbelievable as it may sound, it turns out that this has been the easiest case where the method works. Since Wiles' proof, the method has been refined in order to solve more Diophantine equations, with particular emphasis in the generalized Fermat equations. However, in this occasion it will be enough to apply it in its more primitive version. First, we will review some standard results and terminology.
	
	Let $f$ be a (cuspidal normalized) newform of weight 2 without character on $\Gamma_0(N_f)$ with $q$-series
	\begin{equation}\label{eq:q-series}
		f=q+\sum_{n\ge2}a_n(f)q^n
	\end{equation} 
	and coefficient field $K_f$. 
	Let $E$ be a rational elliptic curve of conductor $N_E$. 
	Suppose that there exists a rational prime $p$ and a prime ideal $\id{p}$ in $K_f$ dividing $p$ such that 
	\begin{equation}\label{eq:cong}
		\bar{\rho}_{E,p}\simeq\bar{\rho}_{f,\id{p}},
	\end{equation}
	where $\bar{\rho}_{E,p}$ and $\bar{\rho}_{f,\id{p}}$ are the residual Galois representation naturally attached to $E$ and $f$, respectively.  
	\mycomment{Recall that (\ref{eq:cong}) is equivalent to have 
		\begin{equation}\label{eq:equiv}
			a_\ell(E)\equiv a_\ell(f)\pmod {\id{p}}
		\end{equation}
		for almost all primes $\ell$.
		
		Note that if the coefficients $a_q(E)$ and $a_q(f)$ are different then we get from (\ref{eq:equiv}) a bound for $p$.} 
	
	In our case, the prime $p$ will be the exponent of equation (\ref{eq:k=5}). The goal of the modular method is to bound this exponent. To do this, we will adopt Mazur's idea, which consist of  applying the following result, based on Serre's work \cite{Serre}.
	
	\begin{prop}\label{prop:mazur}
		Let $E$ be a rational elliptic curve of conductor $N_E$ and let $f$ be a newform of weight 2 an level $N_f\mid N_E$ as in $(\ref{eq:q-series})$. Suppose that there exist a rational prime $p$ and a prime $\id{p}$ in $K_f$ dividing $p$ such that $\bar{\rho}_{E,p}\simeq\bar{\rho}_{f,\id{p}}$. Then for all primes $\ell$ it follows that:
		\begin{itemize}
			\item if $\ell\nmid pN_fN_E$ then $a_\ell(E)\equiv a_\ell(f)\pmod{\id{p}}$,
			\item if $\ell\nmid N_f$ and $\ell\mid\mid N_E$ then $a_\ell(f)\equiv\pm(\ell+1)\pmod {\id{p}}$.
		\end{itemize}
	\end{prop}
	Note that in both cases of Proposition \ref{prop:mazur} each side of the congruence can be computed, and if they happen to be different, we get a
	finite list of candidates for the prime $p$ dividing the norm of their difference. Since this holds for any value
	of $\ell$, in practice, we vary $\ell$ over all primes up to 100 and take the greatest common divisor of the different bounds to reduce the possible values of $p$. Since we are assuming $E$ to be rational, then in particular this allows to bound $p$ when $K_f\neq\Q$.
	
	The idea of the modular method is the following: firstly we will attach to a putative solution of our Diophantine equation (in this case (\ref{eq:k=5})) a rational elliptic curve $E$, usually called \textit{Frey elliptic curve}, of conductor $N$. Then, by modularity results due to Wiles, Breuil, Conrad, Diamond and Taylor \cite{Wiles,TaylorWiles,BreuilConradDiamondTaylor} this will give rise a weight 2 newform of level $N$ such that its $L$-series  equals that of $E$. Hence by  irreducibility theorems for mod $p$ representations of elliptic curves due to Mazur \cite{Mazur} and Ribet's level-lowering theorem \cite[Theorem 1.1]{RibetLowering}, we may conclude that there exists a newform $f$ satisfying the conditions of Proposition~\ref{prop:mazur}. The core of the idea is that at this point, $f$ should be in a space which depends neither on the solution nor on $p$. Then, computing this space and applying Proposition~\ref{prop:mazur} for each of the newforms lying there we could get a bound for $p$.
	
	\subsection{The symplectic argument}\label{sec:symp}
	When following the modular approach, one of the most challenging steps is often when we need to bound the exponent, by giving a contradiction of an isomorphism of residual Galois representations as in (\ref{eq:cong}). This happens, for instance, when both sides of the congruence in Proposition \ref{prop:mazur} are equal.  A symplectic criterion is precisely a very important tool to deal with this problem, originally developed by Kraus and Oesterl\'e \cite{KrausOsterle}. We refer to \cite{FreitasKraus} for a complete list of the latest results.

	Let $f$ be a weight 2 newform without character on $\Gamma_0(N_f)$ satisfying an isomorphism as in (\ref{eq:cong}). Recall that when $K_f\neq\Q$, by Proposition \ref{prop:mazur} we will be able to discard $f$ (for $p$ large enough). Then, we will restrict our attention to the case $K_f=\Q$. By the Eichler-Shimura construction, we can attach to $f$ a rational elliptic curve $E_f$ of conductor $N_f$. In particular, by (\ref{eq:cong}), we get an isomorphism of residual Galois representations of two elliptic curves. 
	
	The idea of the symplectic argument is to consider not only when two elliptic curves have isomorphic residual
	representations, but add information on how the isomorphism relates their Weil's pairings. Let us state one symplectic criterion of Kraus and Oesterl\'{e} that will be used in the present article. 
	
	\begin{prop}[\cite{KrausOsterle}, Proposition 2]
		\label{prop:symplecticmult}
		Let $E, E'$ be elliptic
		curves over $\Q$ with minimal discriminants $\Delta$, $\Delta'$. Let $p$  be a prime such that $\bar{\rho}_{E,p}\simeq\bar{\rho}_{E',p}$. Suppose that
		$E$ and $E'$ have multiplicative reduction at a prime $\ell\neq p$ and that $p\nmid v_{\ell}(\Delta)$. Then $p \nmid v_\ell(\Delta')$ and the representations $\bar{\rho}_{E,p}$ and $\bar{\rho}_{E',p}$
		are symplectically isomorphic
		if and only if  $v_\ell(\Disc)/v_\ell(\Disc')$ is a square modulo $p$.	Furthermore, both representations cannot be symplectically and
		anti-symplectically isomorphic.
	\end{prop}
	
	\subsection{Quadratic forms and Frey elliptic curves}
	\label{subsec:quadraticforms}
	Let $p(x,y)=a_1x^2+a_2xy+a_3y^2$ be a quadratic form. In this short subsection, we see how to
	to attach a Frey elliptic curve to the system 
	\[\begin{cases}
		x=A'a^p,\\
		p(x,y)=B'b^p.
	\end{cases}\]
	
	To easier the notation (this is not necessary but it will apply to our case) suppose that $\gcd(a_1,a_2a_3)=1$ and let $C=\gcd(a_2,a_3)$. Let us consider the identity
	
	\begin{equation}\label{eq:quadraticformx}
		(a_2-4a_1a_3)x^2+4a_3p(x,y)=(a_2x+2a_3y)^2.
	\end{equation}
	Let  $A=A'^2\left(\frac{a_2-4a_1a_3}{C}\right)$, $B=B'\frac{4a_3}{C}$ and $c=\frac{a_2x+2a_3y}{C}$. Replacing equation (\ref{eq:quadraticformx}) in the above system we get
	\begin{align}\label{eq:bennett}
		Aa^{2p}+Bb^p=Cc^2.
	\end{align}
	%and symmetrically 
	%\begin{equation}\label{eq:quadraticformy}
	%	4a_1p(x,y)=(2a_1x+a_2y)^2-(a_2-4a_1a_3)y^2
	%\end{equation}
	Equation (\ref{eq:bennett}) is a ternary Diophantine equation of signature $(2p,p,2)$.  It can be studied by following the modular method, attaching a Frey curve according to the recipes of \cite{BennetSkinner}.
	
	Note that, symmetrically, we can consider the identity 
	\begin{equation}\label{eq:quadraticformy}
		(a_2-4a_1a_3)y^2+4a_1p(x,y)=(a_2y+2a_1x)^2.
	\end{equation}
	Again, setting $C=\gcd(a_1,a_2)$, $A=A'^2\left(\frac{a_2-4a_1a_3}{C}\right)y^2$, $B=B'\frac{4a_1}{C}$ and $c=\frac{a_2y+2a_1x}{C}$, we can replace equation (\ref{eq:quadraticformy}) in the above system to get the following ternary Diophantine equation of signature $(p,p,2)$
	\begin{equation}\label{eq:secondternary}
		A\cdot1^p+Bb^p=Cc^2.
	\end{equation}
	Note that in this case $A$ depends on $y$. Then, in order to attach a Frey elliptic curve to (\ref{eq:quadraticformy}) we must assume some condition on the set of primes dividing $y$. This will be applied in a simple way in Theorem \ref{thm:d=3.2}.
	
	\section{On the equation (\ref{eq:k=5})}
	\label{sec:OnTheEquation}
	\begin{lemma} \label{lemma:samirexpression} Let $r$ be a non-zero integer and $k, d \ge 1$. Let $B_k$ denote the $k$-th  Bernoulli polynomial. Then
		\begin{equation*}
			(x-dr)^k+\cdots+x^k+\cdots+(x+dr)^k=\frac{r^k}{k+1}\left(B_{k+1}\left(\frac{x}{r}+d+1\right)-B_{k+1}\left(\frac{x}{r}-d\right)\right).
		\end{equation*}
		In particular, if $\ell_3=\gcd(d(d+1),3)$, for $k=5$ we get
		\begin{equation}\label{eq:choclogeneral}
			(x-dr)^5+\cdots+x^5+\cdots+(x+dr)^5=\frac{\ell_3(2d+1)}{3}x\left(\frac{3}{\ell_3}x^4+\frac{10d(d+1)}{\ell_3}x^2r^2+\frac{d(d+1)(3d^2+3d-1)}{\ell_3}r^4\right),
		\end{equation}
	\end{lemma}
	\begin{proof}
		See \cite[Lemma 2.1]{PatelSiksek}.
	\end{proof}

	Let $r$ be a non-zero integer and  $d\ge1$. Let us consider the equation
	\begin{equation}\tag{\ref{eq:k=5}}
		(x-dr)^5+\cdots+x^5+\cdots+(x+dr)^5=y^p, \quad \gcd(x,r)=1.
	\end{equation}
	Let us fix some notation:
	\begin{itemize}
		\item $M_d=\max\{v_q(d(d+1)(2d+1)(3d^2+3d-1)): q \text{ prime}\}$,
		\item $\ell_3=\gcd(3,d(d+1))$. \item $\ell_5=\gcd(25,3d^2+3d-1)$.
		%\item From now on, $q$ will be always denote a prime number.
		%\item $\kappa_q=\gcd(x,q)$.
		\item $\kappa_q=\gcd(x,q)$,
		\item $T=\{q \text{ prime} : q\mid d(d+1)\}$, 
		\item  $R=\{q \text{ prime} : q\mid 3d^2+3d-1\}$,
		\item $Q=\{q \text{ prime} : q\mid 2d+1\}$.
	\end{itemize}
	%From now on, $q$ will be always denote a prime number and $\kappa_q=\gcd(x,q)$.
	Let us consider the quadratic form \[p(x_1,x_2)=\frac{3}{\ell_3}x_1^2+\frac{10d(d+1)}{\ell_3}x_1x_2+\frac{d(d+1)(3d^2+3d-1)}{\ell_3}x_2^2.\] By equations (\ref{eq:choclogeneral}) and (\ref{eq:k=5}) we get
	\begin{equation}\label{eq:reducida}
		\frac{\ell_3(2d+1)}{3}xp(x^2,r^2)=y^p, \quad \gcd(x,r)=1.
	\end{equation}
	If $xy=0$ then from (\ref{eq:reducida}) the solution is trivial so from now on we will be interested in non-trivial solutions. Since $\gcd(x,r)=1$,  we have control of the common divisors of the factors on the left hand side of equation (\ref{eq:reducida}). This will be the key to construct a ternary equation of signature $(4p,p,2)$. To be more explicit, the following result holds.

	\begin{prop}\label{prop:powers} Let $(x,y)$ be a non-trivial solution of $(\ref{eq:k=5})$, where $p\ge 2M_d$. There exist non-zero coprime integers $a$ and $b$ such that
		\begin{equation*}
			x=\prod_{q\in T}\kappa_q^{\varepsilon_q} \cdot \prod_{q\in R} \kappa_q^{\delta_q}\cdot\prod_{q\in Q} \kappa_q^{p-v_q(2d+1)+v_q(3)}  \cdot a^p, 
		\end{equation*}
		and
		\begin{equation*}
			p(x^2,r^2)=\prod_{q\in T} \kappa_q^{p-\varepsilon_q} \cdot \prod_{q\in R} \kappa_q^{p-\delta_q}\cdot\prod_{q\in Q} \left(\frac{q}{\kappa_q}\right)^{p-v_q(2d+1)+v_q(3)} \cdot b^p,
		\end{equation*}
		where
		
		\[\varepsilon_q=\begin{cases}
			\frac{v_q(d(d+1))-v_q(3)}{4} & \text{if }  4v_q(x)=v_q(d(d+1))-v_q(3),\\
			p-v_q(d(d+1))+v_q(3) & \text{if } 4v_q(x)>v_q(d(d+1))-v_q(3),
		\end{cases} \]
		and 
		\[\delta_q=\begin{cases}
			\frac{v_q(3d^2+3d-1)-v_q(5)}{2} & \text{if }  2v_q(x)=v_q(3d^2+3d-1)-v_q(5),\\
			p-v_q(3d^2+3d-1)& \text{if }  2v_q(x)>v_q(3d^2+3d-1)-v_q(5).
		\end{cases} \]
		
	\end{prop}
	\begin{proof}
		It follows from Lemmas \ref{lemma:d(d+1)}--\ref{lemma:2d+1} below.
	\end{proof}

	First, we state the following basic result relating the factors on the right hand side of (\ref{eq:choclogeneral}).

	\begin{lemma}\label{lemma:basic}
		Let $d\ge1$. Then
		\begin{itemize}
			\item $\gcd(2d+1,d(d+1))=1$.
			\item $\gcd(3d^2+3d-1,d(d+1))=1$.
			\item $\gcd(3d^2+3d-1,2d+1)$ is supported in $\{7\}$. Moreover, it is equal to 1 if and only if $d\not\equiv3\pmod7$.
		\end{itemize}
	\end{lemma}
	\begin{proof}
		It follows from basic arithmetic.
	\end{proof}

	\begin{lemma}\label{lemma:d(d+1)} Let $p\ge \max\{5,M_d\}$ and $q$ be prime numbers. Let $(x,y)$ be a solution of $(\ref{eq:k=5})$. Suppose that $q\mid d(d+1)$ and $q\mid x$. Then $4v_q(x)\ge v_q(d(d+1))-v_q(3)$. This implies the following  possibilities:
		\begin{itemize}
			\item $v_q(x)=\frac{v_q(d(d+1))-v_q(3)}{4}$ and $v_q(p(x^2,r^2))=pv_q(y)-\frac{v_q(d(d+1))-v_q(3)}{4}$, or
			\item $v_q(x)=pv_q(y)-v_q(d(d+1))+v_q(3)$ and $v_q(p(x^2,r^2))=v_q(d(d+1))-v_q(3)$.			
		\end{itemize}
	\end{lemma}
	\mycomment{
		\begin{lemma} Let $p\ge \max\{5,M_d\}$ and $q\neq3$ be prime numbers. Suppose that $q\mid d(d+1)$ and $q\mid x$. Then $4v_q(x)\ge v_q(d(d+1))$. This implies the following  possibilities:
			\begin{itemize}
				\item $v_q(x)=pv_q(y)-v_q(d(d+1))$ and $v_q(D)=v_q(d(d+1))$, or
				\item $v_q(x)=\frac{v_q(d(d+1))}{4}$ and $v_q(D)=pv_q(y)-\frac{v_q(d(d+1))}{4}$.

			\end{itemize}
		\end{lemma}
	}
	\begin{proof}
		By Lemma \ref{lemma:basic}, $q\nmid(2d+1)(3d^2+3d-1)$. For the sake of contradiction, let us assume that  $4v_q(x)< v_q(d(d+1))-v_q(3)$. By equation (\ref{eq:choclogeneral}) we get
		\[5v_q(x)=pv_q(y).\]
		Since $q\mid x$ and $p\mid v_q(x)$ then $p\le v_q(x)< M_d$, a contradiction. That is, $4v_q(x)\ge v_q(d(d+1))-v_q(3)$. Analyzing  when the equality holds or not,  the rest of the statement follows from equation (\ref{eq:choclogeneral}) and Lemma  \ref{lemma:basic}.
	\end{proof}
	
	\mycomment{
		
		\begin{lemma}
			Let $p\ge \max\{5,M_d\}$ and $q\neq2,5,7$ be prime numbers. Suppose that $q\mid (3d^2+3d-1)$ and $q\mid x$. Then $2v_q(x)\ge v_q(d(d+1))$. This implies the following two possibilities:
			\begin{itemize}
				\item $v_q(x)=\frac{v_q(3d^2+3d-1)}{2}$ and $v_q(D)=pv_q(y)-\frac{v_q(3d^2+3d-1)}{2}$, or
				\item $v_q(x)=pv_q(y)-v_q(3d^2+3d-1)$ and $v_q(D)=v_q(3d^2+3d-1)$.		
			\end{itemize}
		\end{lemma}
	}
	\begin{lemma}
		Let $p\ge 2M_d$ and $q$ be prime numbers. Let $(x,y)$ be a solution of $(\ref{eq:k=5})$. Suppose that $q\mid 3d^2+3d-1$ and $q\mid x$. Then $2v_q(x)\ge v_q(3d^2+3d-1)-v_q(5)$. This implies the following two possibilities:
		\begin{itemize}
			\item $v_q(x)=\frac{v_q(3d^2+3d-1)-v_q(5)}{2}$ and $v_q(p(x^2,r^2))=pv_q(y)-\frac{v_q(3d^2+3d-1)+v_q(5)}{2}+v_q(2d+1)$, or
			\item $v_q(x)=pv_q(y) -v_q(3d^2+3d-1)-v_q(2d+1)$ and $v_q(p(x^2,r^2))=v_q(3d^2+3d-1)$.		
		\end{itemize}
	\end{lemma}
	\begin{proof}
		By Lemma \ref{lemma:basic}, $q\nmid d(d+1)$. Moreover, $q\mid 2d+1$ if and only if $d\equiv 3 \pmod7$ and $q=7$. Also note that $q\neq 2$, since $3d^2+3d-1$ is always odd. For the sake of contradiction, let us assume that  $2v_q(x)< v_q(3d^2+3d-1)-v_q(5)$. By equation (\ref{eq:choclogeneral}) we get
		\[3v_q(x)+v_q(5)+v_q(2d+1)=pv_q(y).\]
		Since $q\mid y$ then $p\le 3v_q(x) +v_q(5)+v_q(2d+1)= (2v_q(x)+v_q(5))+(v_q(x)+v_q(2d+1))\le 2M_d$, a contradiction. The rest of the statement follows from equation (\ref{eq:choclogeneral}) and Lemma 
		\ref{lemma:basic}.
	\end{proof}
	
	\mycomment{
		\begin{lemma} Let $p\ge 4M_d$ and $q\in\{2,5\}$ be prime numbers. Suppose that $q\mid (3d^2+3d-1)$ and $q\mid x$. Then $2v_q(x)\ge v_q(d(d+1))-1$. This implies the following two possibilities:
			\begin{itemize}
				\item $v_q(x)=\frac{v_q(3d^2+3d-1)-1}{2}$ and $v_q(D)=pv_q(y)-\frac{v_q(3d^2+3d-1)-1}{2}$, or
				\item $v_q(x)=pv_q(y)-v_q(3d^2+3d-1)$ and $v_q(D)=v_q(3d^2+3d-1)$.		
			\end{itemize}
		\end{lemma}
	}
	
	\begin{lemma}\label{lemma:2d+1} Let $q\neq 7$ be a prime such that $q\mid2d+1$. Let $(x,y)$ be a solution of $(\ref{eq:k=5})$. Then
		\begin{itemize}
			\item $v_q(x)=0$ and $v_q(p(x^2,r^2))=pv_q(y) -v_q(2d+1)+v_q(3)$, or
			\item $v_q(x)=pv_q(y)-v_q(2d+1)+v_q(3)$ and $v_q(p(x^2,r^2))=0$.		
		\end{itemize}
	\end{lemma}
	
	\begin{proof}
		It follows directly from Lemma \ref{lemma:basic} and equation (\ref{eq:choclogeneral}).
	\end{proof}
	\mycomment{
		\begin{lemma} Let $p\ge M_d$ and $q$ be prime numbers. Suppose that $5v_q(x)<v_q(d(d+1)(3d^2+3d-1))$. Then $q\nmid x$.\TODO{Fijarse si esto se puede escribir bien.. Guiarse con el ejemplo de $d=2$ y $q=17$.}
		\end{lemma}
		\begin{proof}
			Let $q$ be a prime number. If $q \nmid d(d+1)(3d^2+3d-1)$ the statement is trivially true, so we assume $q\nmid d(d+1)(3d^2+3d-1)$ and for the sake of contradiction, let us assume that $q\mid x$. 
			
			\vspace{5pt}
			
			(i) Suppose that $q\mid d(d+1)$. Then by Lemma \ref{lemma:basic}, $q\nmid(2d+1)(3d^2+3d-1)$.
			Since $4v_q(x)+1<v_q(d(d+1))$ then by (\ref{eq:choclogeneral}) we get
			\[5v_q(x)=pv_q(y).\]
			Hence, $p\le pv_q(y)= 5v_q(x)<M_d$, giving a contradiction.
			
			\vspace{5pt}
			
			(ii) Suppose that $q\mid3d^2+3d-1$. Then by Lemma \ref{lemma:basic}, $q\nmid d(d+1)$. Since $2v_q(x)<v_q(3d^3+3d-1)$ then by (\ref{eq:choclogeneral}) we get
			\[v_q(2d+1)+3v_q(x)=pv_q(y).\]
			Hence $p\le pv_q(y)= v_q(2d+1)+3v_q(x) < M_d$, giving a contradiction.
		\end{proof}
	}
	\mycomment{
		Let $S_d$ the set of primes dividing $d(d+1)(2d+1)(3d^2+3d-1)$. Then, by Lemma **, there exist coprime integers $a$ and $b$ such that
		
		\begin{equation}\label{eq:aPowers}
			x=c_1\cdot a^p
		\end{equation}
		and
		\begin{equation}\label{eq:bPowers}
			\frac{3}{\ell_3}x^4+\frac{10d(d+1)}{\ell_3}x^2r^2+\frac{d(d+1)(3d^2+3d-1)}{\ell_3}r^4=c_2 \cdot b^p, 
		\end{equation}
		where $c_i$ is supported in $S_d$. 
	}

	Now we are ready to construct a ternary Diophantine equation of signature $(4p,p,2)$ from a putative solution $(x,y)$ of~(\ref{eq:k=5}). Let
	\begin{gather*}
		S_\varepsilon=\{q\in T : 4v_q(x)>v_q(d(d+1))-v_q(3)\}, \ L_\varepsilon=\{q\in T :  4v_q(x)=v_q(d(d+1))-v_q(3)\}, \\    
		S_\delta=\{q\in R  : 2v_q(x)>v_q(3d^2+3d-1)-v_q(5)\}, \ L_\delta=\{q\in R  : 2v_q(x)=v_q(3d^2+3d-1)-v_q(5)\}.
	\end{gather*}
	We have to deal with some tedious definitions in order to state our result, but it is better to do it for once here. Let
	
	\begin{gather}\label{eq:definitions}
		A_1=\prod_{q\in S_{\varepsilon}}\kappa_q^{4p-5(v_q(d(d+1))-v_q(3))}\prod_{q\in S_\delta}\kappa_q^{4p-6v_q(3d^2+3d-1)+v_q(5)}\prod_{q\in Q}\kappa_q^{4(p-v_q(2d+1)+v_q(3))}, \nonumber\\B_1=\prod_{q\in L_\varepsilon}\kappa_q^{p-\frac{5(v_q(d(d+1))-v_q(3))}{4}}\prod_{q\in L_{\delta}}\kappa_q^{p-\frac{3(v_q(3d^2+3d-1)-v_q(5))}{2}}\prod_{q\in Q}\left(\frac{q}{\kappa_q}\right)^{p-v_q(2d+1)+v_q(3)}, \nonumber\\\label{eq:definitions}
		A=A_1\frac{25d(d+1)-3(3d^2+3d-1)}{\ell_3\ell_5},\quad
		B=B_1\frac{3d^2+3d-1}{\ell_5\prod_{q\in R}\kappa_q^{v_q(3d^2+3d-1)-v_q(5)}},\\
		C=C_1\frac{d(d+1)}{\ell_3\prod_{q\in T} \kappa_q^{v_q(d(d+1))-v_q(3)}}, \quad
		c=\frac{5x^2+(3d^2+3d-1)r^2}{c_1\cdot\prod_{q\in R}\kappa_q^{v_q(3d^2+3d-1)-v_q(5)}},\nonumber
	\end{gather}
	where
	\[ C_1=\begin{cases}
		1 & \text{if } \ \ell_5\in\{1,25\},\\
		\frac{5}{\kappa_5} & \text{if }  \ \ell_5=5,
	\end{cases}\quad c_1=\begin{cases}
		1 & \text{if }  \  \ell_5=1,\\
		5 & \text{if }  \  \ell_5\in\{5,25\}.\end{cases} \]
	\mycomment{
		, \quad  c_1=\begin{cases}
			1 & \text{if }  \  \ell_5=1,\\
			5 & \text{if }  \  \ell_5\in\{5,25\}\end{cases}\\
		c=\frac{5x^2+(3d^2+3d-1)r^2}{c_1\cdot\prod_{q\mid 3d^2+3d-1}\kappa_q^{v_q(3d^2+3d-1)-v_q(5)}},
	}

	\begin{prop}\label{prop:ternaryeq}
		Let $(x,y)$ be a non-trivial solution of 
		$(\ref{eq:k=5})$, where $p\ge2M_d$. Let $A, B, C$ and $c$ as in $(\ref{eq:definitions})$. 
		Let $a$ and $b$ as in Proposition \ref{prop:powers}. Then
		\begin{equation}\label{eq:4pp2}
			Aa^{4p}+Bb^p=Cc^2,
		\end{equation}
		with $Aa, Bb$ and $Cc$ pairwise coprime.
		
	\end{prop}
	\begin{proof}
		By equation (\ref{eq:reducida}) and Proposition \ref{prop:powers} there exist explicit integers $A', B'$ and non-zero coprime integers  $a, b$ satisfying the following system:
		\[\begin{cases}
			x=A'a^p,\\
			p(x^2,r^2)=B'b^p.
		\end{cases}\]
		Then, the result follows applying the argument 
		of Subsection \ref{subsec:quadraticforms}, and reducing the terms.
	\end{proof}
	
	\begin{remark}\label{rem:conductorcurve}
		Following \cite{BennetSkinner} we can attach to (\ref{eq:4pp2}) a Frey elliptic curve $E$, of conductor $N_E$ dividing $N_d\cdot\rad(ab)$ (see Lemma 2.1 of \textit{loc.\ cit.}), where 
		\begin{equation}\label{eq:Nd}
			N_d=2^6\cdot5^2\cdot d^2(d+1)^2\cdot\rad_5\left((2d+1)(3d^2+3d-1)(16d^2+16d+3)\right).
		\end{equation}
		By modularity of $E$ \cite{Wiles,TaylorWiles,BreuilConradDiamondTaylor}, irreducibility of $\bar{\rho}_{E,p}$ \cite[Corollary 3.1]{BennetSkinner} and Ribet's level-lowering theorem \cite[Theorem 1.1]{RibetLowering}, there is a weight 2 newform $f$ without character on $\Gamma_0(N_f)$ such that $\bar{\rho}_{E,p}\simeq\bar{\rho}_{f,\id{p}}$, where $\id{p}$ is a prime in the coefficient field $K_f$ diving $p$ and $N_f$ is an integer dividing $N_d$. Moreover, $N_f$ is easily computable for a particular value of $d$.
		
	\end{remark}

	%\begin{remark}\label{rem:conductorcurve}
	%   Let $p\ge2M_d$ be a prime number and $(x,y)$ a solution of (\ref{eq:k=5}). We see in Proposition \ref{prop:ternaryeq} that this induces a ternary equation of signature $(p,p,2)$. Then following \cite{BennetSkinner} we can construct a Frey elliptic curve $E$, of conductor dividing $2^6 \cdot5^2\cdot\rad_{5}(d(d+1))^2\cdot\rad_5(3d^2+3d-1)\cdot\rad_{35}(2d+1)$ (see Lemma 2.1 of \textit{loc.\ cit.}).
	%\end{remark}
	
	\mycomment{
		Then by equation (\ref{eq:quadraticformx}) we have that
		
		\begin{equation}\label{eq:pp2}
			\begin{split}
				(3d^2+3d-1)&(D)=\\
				&=\frac{d(d+1)}{\ell_3}\left(5x^2+(3d^2+3d-1)r^2\right)^2-\left(\frac{25d(d+1)-3(3d^2+3d-1)}{\ell_3}\right)x^4.
			\end{split} 
		\end{equation}
		Let $a$ and $b$ be as in Proposition \ref{lemma:5powersexpression} and let $c=5x^2+(3d^2+3d-1)r^2$. From  equation (\ref{eq:pp2}) and Proposition \ref{prop:powers}, if $p\ge5M_d$  there exist explicit constants $A, B$ and $C$ depending only on $d$ such that 
		\[Aa^{4p}+Bb^p=Cc^2.\]
		The equation above is a ternary Diophantine equation of signature $(4p,p,2)$.  It can be studied by following the modular method, attaching a Frey curve according to the recipes of \cite{BennetSkinner}. 
	}
	\mycomment{
		\begin{equation*}
			\frac{3d^2+3d-1}{\ell}\cdot c_1\cdot b^p=\frac{d(d+1)}{\ell_3}\frac{25}{\ell}\left(x^2+\frac{3d^2+3d-1}{5}r^2\right)^2-\left(\frac{25}{\ell}\frac{d(d+1)}{\ell_3}-\frac{3}{\ell_3}\frac{3d^2+3d-1}{\ell}\right)\cdot c_1^4\cdot a^{4p}.
		\end{equation*}
	}
	
	\section{Applications}
	\label{sec:applications}
	\subsection{The case $d\equiv1, 7\pmod9$}
	
	Let $$\mu(n)=n\prod_{ \ell\mid n \atop
		\ell \text{ prime}}\left(1+\frac{1}{\ell}\right)$$ and $g_0^+(N)$ be the dimension of the weight 2 newforms on $\Gamma_0(N)$. Recall that $g_0^+(N)\le\frac{N+1}{2}$, by \cite[Theorem 2]{Martin}. In this case the result is as follows.

	\begin{thm}\label{thm:dcong}
		Let $r$ be a non-zero integer,  $d\equiv 1,7\pmod9$ and $p\ge 2M_d$ be a prime number. Let $N_d$ be as in $(\ref{eq:Nd})$ and suppose that $$p>\left(\sqrt{\frac{\mu(N_d)}{6}}+1\right)^{\frac{N_d+1}{6}}.$$  Then, there are no non-trivial integer solutions of
		\[(x-dr)^5+\cdots+x^5+\cdots+(x+dr)^5=y^p, \quad \gcd(x,r)=1.
		\]
	\end{thm}
	\begin{proof}
		Let $(x,y)$ be a non-trivial solution of (\ref{eq:k=5}). %Let $A, B, C, a, b, c$   
		% be as in Proposition \ref{prop:ternaryeq}, satisfying 
		%   \[Aa^{4p}+Bb^p=Cc^p,\]
		%with $\{Aa,Bb,Cc\}$ coprimes two by two.  Following \cite{BennetSkinner} we can construct a Frey elliptic curve $E$, of conductor $N_E$ dividing $N_d\cdot\rad(ab)$ (see Lemma 2.1 of \textit{loc.\ cit.}). 
		Let $E$ be the elliptic curve constructed from this, as explained in Remark \ref{rem:conductorcurve}. Let also $f, N_f$ and $N_d$ as in Remark \ref{rem:conductorcurve}.
		
		Suppose that $p>\left(\sqrt{\frac{\mu(N_d)}{6}}+1\right)^{\frac{N_d+1}{6}}$. We will prove that $K_f=\Q$, following the same strategy as in \cite{KrausB}. For the sake of contradiction, suppose that $K_f\neq \Q$. Then, by Lemme 9 of \textit{loc.\ cit.} there exists a prime $\ell\le\frac{\mu(N_f)}{6}$ such that $a_\ell(f)\notin\Z$. Since $f$ is normalized, then $a_q(f)\in\{0,\pm1\}$ for all $q\mid N_f$, so $\ell\nmid N_f$. On the other hand, since $E$ is rational, then $a_\ell(E)\in\Z$. Moreover, by Hasse's bound and Ramanujan-Petersson conjecture, both $|a_\ell(E)|$ and $|\sigma(a_\ell(f))|$ are bounded by $2\sqrt{\ell}$, for any embedding $\sigma: K_f\to\RR$. Hence, from both cases of Proposition \ref{prop:mazur} we get
		\[p\le (\ell+1+2\sqrt{\ell})^{[K_f:\Q]}\le(\sqrt{\ell}+1)^{2g_0^+(N_f)}\le\left(\sqrt{\frac{\mu(N_d)}{6}}+1\right)^{\frac{N_d+1}{6}},\]
		giving a contradiction. Hence we can assume $K_f=\Q$. 
		
		Since $d\equiv 1, 7\pmod 9$ then $v_3(2d+1)=1$ and $\ell_3=\gcd(3,d(d+1))=1$. Then, if $A$ and $B$ are as in Proposition \ref{prop:ternaryeq},  
		$3\mid AB$ and $v_3(A)\equiv v_3(B)\equiv0\pmod p$. In particular, $E$ has multiplicative reduction at $3$ and by Ribet's level-lowering theorem,  $3\nmid N_f$. Then, by Proposition \ref{prop:mazur} we get $a_3(f) \equiv \pm 4\pmod{p}$. %On the other hand, by Hasse's bound, we know that $|a_3(f)|\le2\sqrt{3}$. 
		Since $a_3(f)\neq 4$ (by Hasse's bound) and $p\mid a_3(f)\pm 4$ then $p\le |a_3(f)|+4<8$, giving a contradiction.
	\end{proof}
	
	\subsection{The case $d=2$}
	
	In this case we prove the following.
	
	\begin{thm}\label{thm:d=2}
		Let $r$ be a non-zero integer and $p>7$ be a prime. If 
		\begin{equation}\label{eq:d=2}
			(x-2r)^5+(x-r)^5+x^5+(x+r)^5+(x+2r)^5=y^p, \quad \gcd(x,r)=1,
		\end{equation}	
		then $2\mid x$.
	\end{thm}
	\begin{proof}
		Let $(x,y)$ be a putative non-trivial solution of (\ref{eq:d=2}). 
		By Proposition \ref{prop:ternaryeq}, there exist non-zero coprime integers $a,b$ such that \begin{equation*}\label{eq:ternaryd=2}
			33\cdot\kappa_2^{4p-5}\kappa_5^{4p-4}\kappa_{17}^{4p-6}a^{4p}+\frac{17}{\kappa_{17}}\left(\frac{5}{\kappa_5}\right)^{p-1}b^p=\frac{2}{\kappa_2}\left(\frac{5x^2+17r^2}{\kappa_{17}}\right)^2.
		\end{equation*}
		Let $E$ and $f$ be the elliptic curve and newform constructed as in Remark \ref{rem:conductorcurve}. In this particular case, the level of the newform is $N_f=2^\delta\cdot3\cdot5\cdot11\cdot17$, where
		\[\delta=\begin{cases}
			0 & \text{if } \kappa_2=2,\\
			8 & \text{if } \kappa_2=1.
		\end{cases}
		\]
		
		Since $E$ has a non-trivial rational $2$-torsion point (see \cite[Lemma 2.1]{BennetSkinner}) then $a_\ell(E)\equiv0\pmod2$ for each prime $\ell$ of good reduction, by \cite[Theorem 2]{Katz}.  Using this fact and \verb|Magma|, we run Proposition \ref{prop:mazur} for each of the newforms in $S_2(N_f)$, and we are able to discard all of them for $p>7$, when $N_f=3\cdot5\cdot11\cdot17$. This implies that $2\mid x$, proving the result. On the other hand, when $2\nmid x$ the level $N_f$ is too large, hence the space $S_2(N_f)$ seems to be infeasible to compute.
	\end{proof}
	\begin{remark}
		In particular, if $p>7$ and $2\mid r$, there are no non-trivial solutions of (\ref{eq:d=2}).
	\end{remark}
	
	\subsection{The case $d=3$}
	
	In this case we prove the following results.
	\begin{thm}\label{thm:d=3.1}
		Let $r$ be a non-zero integer and $p>19$ be a prime. If $(x,y)$ is a non-trivial solution of 
		\begin{equation}\label{eq:d=3}
			(x-3r)^5+(x-2r)^5+(x-r)^5+x^5+(x+r)^5+(x+2r)^5+(x+3r)^5=y^p,\quad \gcd(x,r)=1,
		\end{equation}	
		then $2\nmid x$ and $35\mid x$. Moreover, if $\kro{-5}{p}=-1$ and in addition,  $\kro{2}{p}=1$ or $\kro{3}{p}=1$, then there are no non-trivial solutions.
		
	\end{thm}
	\begin{proof}
		Let $(x,y)$ be a hypothetical non-trivial solution of (\ref{eq:d=3}). By Proposition \ref{prop:ternaryeq}, there exist non-zero coprime integers $a,b$ such that 
		\begin{equation*}
			\kappa_2^{4p-10}\cdot\kappa_5^{4p-5}\cdot\kappa_7^{4p-10}\cdot13\cdot a^{4p}+\left(\frac{7b}{\kappa_7}\right)^p=\frac{5}{\kappa_5}\left(\frac{2x^2+14r^2}{\kappa_2\cdot\kappa_7}\right)^2.
		\end{equation*}
		Let $E$ and $f$ be the elliptic curve and newform constructed as in Remark \ref{rem:conductorcurve}. In this particular case, the level of the newform is $N_f=	2^\delta\cdot\frac{5^2}{\kappa_5}\cdot\kappa_7\cdot13$, with
		\begin{equation*}
			\delta=\begin{cases}
				5 & \text{if } \kappa_2=1,\\
				1 & \text{if } \kappa_2=2.
			\end{cases}
		\end{equation*}
		\mycomment{
			
			\begin{equation*}
				N_f=\begin{cases}
					2^5\cdot5^\delta\cdot13, & \text{in case 1} \\
					2^5\cdot5^\delta\cdot7\cdot13, & \text{in case 2,}\\
					2\cdot5^\delta\cdot13, & \text{in case 3,}	\\
					2\cdot5^\delta\cdot7\cdot13			& \text{in case 4,}	
				\end{cases}
				\quad\text{ where } \delta=\begin{cases}
					2, & \text{if } \kappa=1,\\
					1, & \text{if } \kappa=5.
				\end{cases}
		\end{equation*}}
		
		Since $E$ has a non-trivial rational $2$-torsion point then $a_\ell(E)\equiv0\pmod2$ for each prime $\ell$ of good reduction. On the other hand, if $\kappa_7=1$ we have that $7\nmid N_f$ and $7\mid\mid N_E$ (see  \cite[Lemma 2.1]{BennetSkinner}), so that ${a_7(f)\equiv\pm8\pmod{\id{p}}}$. Using this fact and \verb|Magma|, we run Proposition \ref{prop:mazur} for each of the newforms in $S_2(N_f)$, and we are able to discard all of them if $p>19$, except for five newforms of level $N_f=2^5\cdot5\cdot7\cdot13$. This proves the first part of the statement. 
		
		To discard the remaining five newforms, we will apply the symplectic argument. All five have
		rational coefficients so they correspond to isogeny classes of elliptic curves, namely those with LMFDB \cite{lmfdb} label \textcolor{MidnightBlue}{\href{https://www.lmfdb.org/EllipticCurve/Q/14560/k/}{{\ttfamily 14560.k}}}, \textcolor{MidnightBlue}{\href{https://www.lmfdb.org/EllipticCurve/Q/14560/k/}{{\ttfamily14560.q}}}, \textcolor{MidnightBlue}{\href{https://www.lmfdb.org/EllipticCurve/Q/14560/s/}{{\ttfamily 14560.s}}}, \textcolor{MidnightBlue}{\href{https://www.lmfdb.org/EllipticCurve/Q/14560/j/}{{\ttfamily 14560.j}}}, \textcolor{MidnightBlue}{\href{https://www.lmfdb.org/EllipticCurve/Q/14560/r/}{{\ttfamily 14560.r}}}. Let us follow the notation of Proposition \ref{prop:symplecticmult}. We choose one elliptic curve for each isogeny class (the computation is independent of the choice). Hence we have:

		\begin{enumerate}[(i)]
			\item $E'=14560.k1$, \quad $\Delta'=2^6\cdot5\cdot7^2\cdot13$, or
			\item $E'=14560.q1$, \quad $\Delta'=2^9\cdot5^2\cdot7^3\cdot13^2$, or
			\item $E'=14560.s1$, \quad $\Delta'=2^6\cdot5\cdot7^4\cdot13$, or
			\item $E'=14560.r1$, \quad $\Delta'=2^6\cdot5^3\cdot7^4\cdot13^3$, or
			\item $E'=14560.j3$, \quad $\Delta'=2^6\cdot5^2\cdot7^6\cdot13^4$. 
		\end{enumerate} 
		By \cite[Lemma 2.1]{BennetSkinner}, the Frey curve $E$ has minimal discriminant  $\Delta=2^6\cdot5^{4p-5}\cdot7^{4p-10}\cdot13\cdot(a^4b^2)^p$.
		
		Suppose that $\kro{-5}{p}=-1$. Then, applying Proposition \ref{prop:symplecticmult} for  primes $\ell=5,13$ we reach a contradiction in 
		cases (i)--(iv). However, in (v) we conclude that $E$ and $E'$ are symplectically isomorphic (when $\ell=13$) and that they are anti-symplectically isomorphic if and only if $\kro{-5/2}{p}=-1$ (when $\ell=5$). Since we are under the assumption  $\kro{-5}{p}=-1$, the latter is equivalent to $\kro{2}{p}=1$. Similarly, considering $\ell=7$, we get that $E$ and $E'$ are anti-symplectically isomorphic if and only if $\kro{3}{p}=1$, concluding the proof.
	\end{proof}
	\begin{remark}\label{rem:7|r}
		In particular, if $p>19$ and $7\mid r$ or $5\mid r$ then there are no non-trivial solutions of (\ref{eq:d=3}).
	\end{remark}
	
	\begin{thm}\label{thm:d=3.2}
		Let $p>19$ be a prime and $r=2^{a_1}5^{a_2}7^{a_3}13^{a_4}$ with integers $a_i\ge0$. Then, there are no non-trivial solutions of equation
		\[	(x-3r)^5+(x-2r)^5+(x-r)^5+x^5+(x+r)^5+(x+2r)^5+(x+3r)^5=y^p, \quad \gcd(x,r)=1.\]
	\end{thm}
	\begin{proof}
		Let $r$ be as in the statement and $(x,y)$ be a solution.  For the sake of contradiction assume that $(x,y)$ is non-trivial. If $p>19$,  we get from Theorem~\ref{thm:d=3.1} that $(\kappa_2,\kappa_5,\kappa_7)=(1,5,7)$. Then, from Proposition \ref{prop:powers}, there exist non-zero coprime integers $a, b$ such that 
		\begin{equation}\label{eq:powers}
			x=5^{p-1}\cdot7^{p-2}\cdot a^p \ \text{ and } \ x^4+40r^2x^2+140r^4=5\cdot 7 \cdot b^p.
		\end{equation}
		An easy calculation shows that in this case equation (\ref{eq:secondternary}) is equal to 
		\begin{equation}\label{eq:case2 24n}
			52r^4\cdot 1^p+7b^p=5\left(\frac{x^2+20r^2}{5}\right)^2.
		\end{equation}
		Again, by \cite{BennetSkinner} we can construct a Frey elliptic curve $F$, from (\ref{eq:case2 24n}). Since $r=2^{a_1}5^{a_2}7^{a_3}13^{a_4}$ then we have control on the conductor of $F$. Moreover,  by modularity of rational elliptic curves and Ribet's level-lowering theorem we deduce that there is a weight 2 newform $f$ without character on $\Gamma_0(N_f)$ such that $\bar{\rho}_{F,p}\simeq\bar{\rho}_{f,\id{p}}$, where $\id{p}$ is a prime in the coefficient field $K_f$ diving $p$ and $N_f=2^\delta\cdot5^2\cdot7\cdot13$, with
		\begin{equation*}
			\delta=\begin{cases}
				3 & \text{if } v_2(r)=0,\\
				0 & \text{if }  v_2(r)=1,\\
				1 & \text{if }  v_2(r)\ge2.\\
				
			\end{cases}	
		\end{equation*}
		Once again, using that $F$ has a non-trivial rational $2$-torsion point we apply Proposition \ref{prop:mazur}. We can discard all the possible newforms if $p>11$, giving the desired contradiction.
	\end{proof}

	\bibliographystyle{alpha}
	\bibliography{biblio}
	
\end{document}